\documentclass[a4paper,10pt]{article}
\usepackage[a4paper,left=3.6cm,right=3.6cm]{geometry}

\usepackage[utf8]{inputenc}
\usepackage[T1]{fontenc}
\usepackage{lmodern}
\usepackage[english]{babel}
\usepackage{graphicx}
\usepackage{amsmath}
\allowdisplaybreaks[4]
\usepackage{amsthm}
\usepackage{amssymb}
\usepackage{mathrsfs}
\usepackage{amstext}
\usepackage{mathtools}
\usepackage{tikz}
\newcommand{\achsen}{
	\draw[-{Stealth}] (0,0) -- (1.15,0);
	\draw[-{Stealth}] (0,0) -- (0,1.15);
	\draw (1,-0.015) -- (1,0.015);
	\draw (-0.015,1) -- (0.015,1);
	\node[below left] at (0,0) {$0$};
	\node[below] at (1,-0.015) {$1$};
	\node[left] at (-0.015,1) {$1$};
}

\usepackage{dsfont}
\usepackage{adjustbox}
\usetikzlibrary{quantikz2, angles ,quotes}
\usepackage{enumitem}
\usepackage{mleftright}
\usepackage{csquotes}
\usepackage{authblk}

\usepackage{caption}
\usepackage{subcaption}
\usepackage{bm}

\usepackage{todonotes}
\usepackage[
backend=biber,
maxnames=10,
giveninits=true,
style=numeric-comp,
sorting=nyt,
sortcites=true,
doi=true,
url=false,
isbn=false,
eprint=false,
]{biblatex}
\usepackage{xurl}  
\usepackage[colorlinks=true, 
linkcolor=black,
citecolor=black,
urlcolor=black]{hyperref}

\appto\bibfont{\normalfont}
\newtheorem{thm}{Theorem}[section]
\newtheorem{lem}[thm]{Lemma}

\newtheorem{prop}[thm]{Proposition}

\theoremstyle{definition}
\newtheorem{defi}[thm]{Definition}
\newtheorem{ex}[thm]{Example}

\theoremstyle{remark}

\usepackage{cleveref}

\newcommand{\BS}{{S^r_{p,q}B(\bbT^d)}}

\newcommand{\bbC}{\mathbb{C}}
\newcommand{\bbB}{\mathbb{B}}
\newcommand{\bbN}{\mathbb{N}}
\newcommand{\bbR}{\mathbb{R}}
\newcommand{\bbT}{\mathbb{T}}
\newcommand{\bbZ}{\mathbb{Z}}

\newcommand{\sB}{\mathscr{B}}
\newcommand{\sC}{\mathscr{C}}

\newcommand{\mcV}{\mathcal{V}}
\newcommand{\mcW}{\mathcal{W}}

\newcommand{\mcA}{\mathcal{A}}
\newcommand{\mcK}{\mathcal{K}}

\newcommand{\bj}{\mathbf{j}}
\newcommand{\bk}{\mathbf{k}}
\newcommand{\bl}{\bm{l}}
\newcommand{\bh}{\mathbf{h}}
\newcommand{\lowl}{{\ell_0}}
\newcommand{\highl}{{\ell_1}}

\newcommand{\supp}{\operatorname{supp}}
\newcommand{\diffd}{\mathrm{d}}

\newcommand{\Bra}[1]{\bra{#1}\!}
\newcommand{\Ket}[1]{\!\ket{#1}}

\newcommand{\errquant}{e^\text{qua}}
\newcommand{\errrand}{e^\text{ran}}
\newcommand{\errdet}{e^\text{det}}
\newcommand{\vol}{\text{vol}}

\DeclareMathOperator*{\esssup}{ess\,sup}

\title{Quantum integration in Besov spaces with mixed smoothness}
\author[a,b]{Bernd Käßemodel%
	\thanks{Corresponding author. Email: bernd.kaessemodel@math.tu-chemnitz.de}%
	\textsuperscript{,}}
\author[a]{Tino Ullrich%
	\thanks{Email: tino.ullrich@math.tu-chemnitz.de}%
	\textsuperscript{,}}
\affil[a]{Chemnitz University of Technology, Faculty of Mathematics}
\affil[b]{Leipzig University, Faculty of Mathematics and Computer Science}
\date{\today}

\begin{document}
	\maketitle
	\begin{abstract}
		We study the complexity of high-dimensional integration of periodic functions from  Besov spaces $\BS$ with mixed smoothness in the quantum model of computation. We construct new quantum algorithms based on the multivariate Faber-Schauder decomposition of such functions and analyze their convergence rates. We show a gain of  one convergence order compared to the well-known deterministic rates of convergence. We additionally obtain lower bounds showing the optimality of our results up to logarithms.
	\end{abstract}
	\section{Introduction}
	Quantum algorithms for continuous numerical problems have been studied since the early 2000s, with numerical integration and function approximation as central examples \cite{novak2001,heinrich2003sobolev,heinrich2004approx2,wozni06}. The main tool underlying many of these algorithms is quantum summation and was studied by Heinrich \cite{heinrich2002summation} and by Heinrich and Novak \cite{heinrich2001optimal} and builds on amplitude amplification and amplitude estimation.  These methods generalize Grover's search method and also use ideas from Shor's factoring algorithm \cite{grover1996,shor1997,brassard2000}. 
	
	In his 2002 work on quantum summation, Heinrich established a rigorous quantum query model within the framework of information-based complexity \cite{heinrich2002summation}. This model enables a systematic comparison of numerical quantum query algorithms with deterministic and randomized classical algorithms. Although these foundations were laid more than two decades ago, many questions concerning the quantum complexity of continuous numerical problems remain open.
	
	In this paper, we investigate the problem of approximating the integral
	\begin{align}\label{intontorus}
		I_d(f)\coloneq\int_{[0,1)^d} \!f(x) \,\text{d}x 
	\end{align}
	for periodic functions $f$ belonging to the Besov spaces of dominating mixed smoothness $\BS$ on the $d$-dimensional torus by quantum algorithms that make at most $n$ queries to a quantum oracle.
	
	Quantum integration has previously been studied by Novak for multivariate Hölder classes \cite{novak2001}, by Heinrich for isotropic Sobolev classes \cite{heinrich2003sobolev}, and by Ye and Hu for anisotropic Sobolev classes \cite{ye2008}. More recently, Li and Ye considered generalized Hölder classes with isotropic smoothness \cite{li2026} and Go\'cwin studied univariate integration for piecewise Hölder-regular functions with an unknown singularity \cite{gocwin2023}. A randomized variant of the quantum query model for continuous problems, including integration, was introduced by Woźniakowski \cite{wozni06}. Further algorithmic approaches include quantum-accelerated Monte Carlo methods \cite{montanaro2015,kaneko2021,herbert2022}, circuit-level integration schemes \cite{shimada2020,shu2024}, and a framework based on quantum singular-value transformation \cite{nghiem2025}.
	
	Within the information-based complexity framework, we continue this line of research by considering Besov spaces of dominating mixed smoothness (bounded mixed difference). Integration of functions with mixed smoothness has a long-standing tradition in the classical (deterministic) setting. For overviews of high-dimensional quasi-Monte Carlo integration, of integration in spaces of mixed smoothness, and Faber-Schauder bases, see \cite{dks2013}, \cite[Chapter 8]{dung2018} and \cite[Chapter 3]{triebel2010}, respectively. 
	
	For Besov spaces of dominating mixed smoothness, D\~ung and T.\ Ullrich established general lower bounds in both the periodic and non-periodic settings, see \cite{du2015}, the references therein and Temlyakov \cite{temlyakov1990}. For Besov and Triebel--Lizorkin spaces of dominating mixed smoothness M.\ Ullrich and T.\ Ullrich derived upper bounds for Frolov's cubature formula. These bounds are optimal in the large-smoothness regime, while optimality remains open in part of the small-smoothness regime \cite{uu2016}. However, Frolov's cubature formula does not represent a QMC rule, although the weights are equal (they do not sum up to one).  Hinrichs et al.\ used Faber-Schauder coefficients to construct quasi-Monte Carlo rules that attain the optimal asymptotic order in the deterministic setting for $1/p<r<2$, see \cite{hinrichs2016}. Käßemodel, Nagel and T.\ Ullrich used the tent transform in a related Faber-Schauder-based analysis of integration in non-periodic function spaces \cite{knu2026}. Optimal-order quasi-Monte Carlo rules based on higher-order digital nets and sequences were established in \cite{gsy2016,gsy2017,gsy2018} and almost optimal lattice rules in \cite{gsy2019}. More recently, Goda and Krieg introduced a universal integration algorithm for weighted Korobov spaces \cite{gk2026}.
	
	In the randomized setting (Monte Carlo) better error rates are possible in related settings. For example, optimal asymptotic rates are known for Monte Carlo integration in the mixed Sobolev case \cite{ul17}. General procedures for determining lower bounds in this setting can be found in \cite{bakhvalov2015}.
	
	Together, these results provide classical reference rates for assessing whether quantum algorithms can achieve faster convergence.
	
	We address this question by constructing new quantum algorithms for the integration problem based on two main ingredients. The first is the tensorized Faber–Schauder decomposition of continuous functions
	\begin{align*}
		f=\sum_{\bj \in \bbN_{-1}^d}\sum_{\bk\in \bbB_\bj} d_{\bj,\bk}(f) v_{\bj,\bk},
	\end{align*}
	with coefficients $(d_{\bj,\bk}(f))_{\bj,\bk} \subset \mathbb{C}$ and a tensorized version of scaled and shifted hat-functions $v_{\bj,\bk}$ as basis (see Fig.\ \ref{fig:faberhats}). Together with the  sequence space isomorphism between  $\BS$ and the sequence spaces $s_{p,q}^rb$ given by 
	$$
	\| f \|_{S_{p,q}^r B(\bbT^d)} \asymp \Big(\sum_{\bj \in \bbN_{-1}^d} 2^{|\bj|_1(r-1/p)q}\Big( \sum_{\bk \in \bbB_{\bj}} |d_{\bj,\bk}(f)|^p  \Big)^{q/p} \Big)^{1/q}
	$$
	valid for $1\leq p,q\leq \infty$ and $1/p < r < 1+1/p$, see  \cite{hinrichs2016}, we are able to study the integration problem level-wise.  
	The second ingredient is a multilevel scheme. Using this decomposition, we express the  
	integral as a hierarchical sum over successive resolution levels. The contribution of each level is a scaled finite sum of the corresponding Faber-Schauder coefficients. Each Faber-Schauder coefficient can be recovered from at most $3^d$ point evaluations, independently of the resolution level. In this multilevel scheme the lower levels are computed to arbitrarily high precision by simulating a classical deterministic computation, whereas the higher-level contributions are approximated by quantum summation algorithms using a level-dependent balanced query allocation that controls the total integration error.
	
	Based on this construction, we obtain upper bounds for the $n$-th minimal quantum query error of the integration problem.
	
	\begin{thm}\label{thm:besovintup}
		Let $1\le p,q \le \infty$, $\max\{1/p,2/q-1\}<r<2$. Then, for all $n \in \mathbb N$ with $n > 4$, 
		\begin{align*}
			\errquant_n(I_d,\sB(\BS))        
			\lesssim n^{-r-1} &(\log n)^{(d-1)(r+1-1/q)} \lambda_{p,q}(n).
		\end{align*}
		Here, $\lambda_{p,q}(n)$ denotes the parameter-dependent logarithmic factor:
		\begin{align*}
			\lambda_{p,q}(n)=
			\begin{cases}
				1 & \text{if } \min\{p,q\}>2, \\
				(\log\log n)^{3/2} \log\log\log n & \text{if } \min\{p,q\}=2,\\
				(\log n)^{2/\min\{p,q\}-1}& \text{if } \min\{p,q\} < 2.
			\end{cases}
		\end{align*}
	\end{thm}
	
	Let us emphasize that the above restriction $r<2$ comes from the Faber-Schauder characterization which can be partly extended to $r=2$, see \cite[Thm.\ 3.3]{knu2026}. In fact, using 
	$$
	\sup\limits_{\bj \in \bbN_{-1}^d} 2^{ 3|\bj|_1/2} \left(\sum_{\bk \in \bbB_{\bj}} |d_{\bj, \bk}(f)|^2\right)^{1/2} 
	\lesssim \|f\|_{H^2_{\text{mix}}(\bbT^d)}
	$$
	for any $f\in H^2_{\text{mix}}(\bbT^d)$, the analogous algorithm as above yields a result for mixed second order Sobolev functions.
	\begin{thm}\label{thm:besovintup_2} It holds for $n>4$
		$$
		\errquant_n(I_d,\sB(H^2_{\text{mix}}(\bbT^d)))      
		\lesssim n^{-3} (\log n)^{3(d-1)}(\log\log n)^{3/2} \log\log\log n .
		$$
	\end{thm}
	To assess the optimality of these decay rates, we complement the upper estimates with a lower bound. Using an atomic decomposition of the mixed Besov spaces together with lower bounds for quantum summation, we establish the following lower bound for the $n$-th minimal quantum query error.
	
	\begin{thm}\label{thm:besovintlow}
		Let $1 \le p,q \le \infty$ and $r > 1/p$. Then, for $n\in \bbN, n\ge 2$, we know that
		\begin{align*}
			\errquant_n(I_d,\sB(\BS))
			\gtrsim n^{-r-1}(\log n)^{(d-1)(1/p-1/q)_+}.
		\end{align*}
	\end{thm}
	
	The assumption $r> 1/p$ in Theorem \ref{thm:besovintlow} ensures the continuous embedding 
	\begin{align*}
		\BS \hookrightarrow C(\bbT^d)
	\end{align*}
	so that point evaluations are well-defined. However, the proof itself does not rely on this restriction. For $0<r\le 1/p$, the error estimate in Theorem \ref{thm:besovintlow} remains valid, but should be understood as a bound on the continuous subclass $\sB(\BS)\cap C(\bbT^d)$.  
	
	Consequently, within the parameter range of Theorem \ref{thm:besovintup} considered for the upper bounds, the polynomial rate cannot be improved and the upper and lower estimates differ at most by logarithmic factors.
	
	Next, we compare the quantum bounds with the known error rates in the classical models of computation. 
	In the deterministic setting, the minimal worst-case error using at most $n$ function evaluations, for the same periodic Besov spaces and parameter range satisfies
	\begin{align*}
		\errdet_n(I_d,\sB(\BS)) \asymp n^{-r}(\log n)^{(d-1)(1-1/q)},
	\end{align*}
	see \cite{hinrichs2016}. 
	In the randomized setting, sharp error bounds are known for the mixed Sobolev case on the unit cube, corresponding to $p=q=2$ in the Besov scale. More precisely, for $H_\text{mix}^r = S_{2,2}^rB$ and $r\ge0$
	\begin{align*}
		\errrand_n(I_d,\sB(H_\text{mix}^r([0,1]^d))) \asymp n^{-r-1/2},
	\end{align*}
	see \cite{ul17}, where the space $H_\text{mix}^r$ is denoted by $H_2^r$ and $\errrand_n$ denotes the minimal worst-case root mean square error using at most $n$ function evaluations in expectation for every input. These classical estimates show that, in terms of the polynomial decay, the quantum rate gains one power of $n$ over the deterministic rate and half a power of $n$ over the randomized rate in the mixed Sobolev setting.
	
	\par\medskip
	\textbf{Notation.} As usual $\bbN, \bbZ, \bbR$ and $\bbC$ denote the sets of natural numbers, integers, real numbers and complex numbers. Furthermore, let $\bbN_0=\bbN \cup\{0\}$, $\bbN_{-1}=\bbN_0 \cup\{-1\}$ and $\bbB_j=\bbZ\cap [0,2^j)$ for $j \in \bbN_{-1}$. For a real number $a$ we set $(a)_+\coloneq \max\{a,0\}$, $\lfloor a \rfloor = \max\{ n \in \bbZ:n\le a \}$ and $\lceil a\rceil=-\lfloor-a \rfloor$. The letter $d$ is reserved for the underlying dimension in expressions such as $\bbN^d, \bbR^d$ etc. Let $\chi_A$ denote the indicator function of a set $A$. Multi-indices $\bj=(j_1,\dots,j_d)\in \bbN_{-1}^d$ are denoted by boldface letters and for such $\bj$ we set $|\bj|_1=\sum_{i=1}^{d} |j_i|$. For a normed space $X$, we denote its closed unit ball by $\sB(X)$. For sequences $(x_n),(y_n)$, we write $x_n \lesssim y_n$ if there exists a constant $c>0$ (which may depend on $d,r,p$ and $q$) such that $x_n \le cy_n$ for all $n \in \bbN$. We will write $x_n \asymp y_n$ if $x_n \lesssim y_n$ and $y_n \lesssim x_n$. Throughout, $\log$ denotes the logarithm to base 2.

	\section{The quantum model of computation}
	
	In this section we briefly recall Heinrich's quantum query model for numerical problems \cite{heinrich2002summation}. Within this framework Heinrich and Novak investigated the complexity of summation and integration problems \cite{heinrich2002summation,heinrich2001optimal,heinrich2003sobolev}. The most important results regarding quantum summation are summarized at the end of this section.
	
	For $m \in \bbN$ let $H_m\coloneq \bbC^2 \otimes \dots \otimes\bbC^2$ be the Hilbert space of an $m$-qubit system and let $\sC_m=\{\ket i: i \in \bbB_m\}$ denote its computational basis, where the integers $i$ are identified with their binary representations.

	Within the framework of information-based complexity, a numerical problem is specified by nonempty sets $D$ and $R$, a class $F\subset R^D$ of functions from $D$ to $R$, a normed space $G$ and a mapping $S: F \to G$. To describe quantum algorithms for approximating $S(f)$, $f \in F$ in the norm of $G$, we have to specify how they obtain information about $f$. This is done through the notion of a quantum query: A \emph{(quantum) query $Q$ on $F$} is a tuple 
	\begin{align*}
		Q=(m,m',m'',Z,\tau, \beta)
	\end{align*}
	consisting of integers $m,m',m'' \in \bbN$ such that $m'+m''\le m$, a nonempty subset $Z\subset \bbB_{m'}$ and two mappings 
	\begin{align*}
		\tau &:Z \to D \quad\text{and}\\
		\beta &: R \to \bbB_{m''}.
	\end{align*}
	The mapping $\tau$ determines the sampling points accessible to the algorithm, while $\beta$ provides the encoding of the corresponding function values. The integer $m(Q)\coloneq m$ determines the number of qubits used by $Q$. For a quantum query $Q$, a function $f \in F$, and $\,\Ket i \Ket x \Ket y \in \sC_{m'}\otimes \sC_{m''} \otimes \sC_{m-m'-m''}$ (if $m=m'+m''$, the last register is omitted), we set
	\begin{align*}
		Q_f \Ket i \Ket x \Ket y = 
		\begin{cases}
			\Ket i \Ket {x\oplus (\beta \circ f \circ \tau ) (i)}\Ket y &\textrm{if } i\in Z,\\
			\Ket i\Ket x\Ket y & \textrm{otherwise},
		\end{cases}
	\end{align*}
	where $\oplus$ means addition modulo $2^{m''}$. In this way a unitary operator $Q_f$ on $H_m$ is assigned to $Q$.
	
	A \emph{quantum algorithm from $F$ to $G$} is a tuple $A=(Q,(U_j)_{j=0}^{n},b,\varphi)$ consisting of a quantum query $Q$ on $F$, $n+1 \in \bbN$ unitary operators $U_0,\dots,U_n$ on $H_m$, an integer $b \in \bbB_m$ representing the initial basis state and the output mapping $\varphi: \bbB_m \to G$ , with $m=m(Q)$. The integer $n(A)\coloneq n$ denotes the number of queries of $A$. Throughout this paper, complexity is measured by this number, the applications of the query operator $Q_f$. The number of qubits, the unitary transformations and the postprocessing are not included in the query cost. For $n \in \bbN_0$ let $\mathcal A_n$ denote the set of all quantum algorithms $A$ with $n(A)\le n$.
	
	For each $f \in F$ a quantum algorithm $A=(Q,(U_j)_{j=0}^n, b, \varphi)$ from $F$ to $G$ induces a probability measure $A(f)$ on $G$ in the following way: Set $A_f\coloneq U_nQ_fU_{n-1}\dots U_1Q_fU_0$ and for $C\subset G$
	\begin{align*}
		A(f)(C) \coloneq \sum_{i \in \varphi^{-1}(C)} |\!\Bra i A_f \Ket {b} \!|^2,
	\end{align*}
	where $\Bra{i} A_f \Ket b$, $i \in \bbB_m$ denotes the coefficients of $A_f\Ket{b}$ in the basis representation, i.e.
	\begin{align*}
		A_f\Ket{b} = \sum_{\Ket i \in \sC_m} \Bra {i} A_f \Ket{b} \Ket i \!.
	\end{align*}
	
	The following quantity plays the central role in our investigation. It describes the minimal worst-case error decay of a given numerical problem in the quantum setting. 
	\begin{defi}
		Let $n \in \bbN_0, \theta \in [0,1)$. The \emph{$n$-th minimal query error} is defined as
		\begin{align*}
			\errquant_n(S,F,\theta)\coloneq \inf_{A \in \mcA_n} \sup_{f \in F} \inf \{\varepsilon>0: A(f)(\sB_\varepsilon(S(f))) \ge 1-\theta\},
		\end{align*}
		where $\sB_\varepsilon(S(f))$ denotes the closed ball in $G$ with radius $\varepsilon$ and center $S(f)$. 
	\end{defi}
	
	The above definition of an algorithm is a special case of quantum algorithms with several measurements. Since the present paper only deals with the investigation of asymptotic error rates, we can limit ourselves to algorithms with one measurement without restriction \cite[Lem. 1]{heinrich2002summation}. Moreover, the specific choice of the constant $\theta$ is essentially arbitrary. In our case ($G=\bbR$) any $\theta \in (0,1/2)$ could be used. This was shown in \cite[Lem. 3]{heinrich2002summation}. In view of this we write
	\begin{align*}
		\errquant_n(S,F) \coloneq \errquant_n(S,F,1/4).
	\end{align*}
	
	Finally, we mention two general results on quantum errors and one specific result on the summation problem from Heinrich. The first one gives a way for comparing the quantum query errors of two related numerical problems.  
	
	\begin{lem}[{\cite[Cor. 1]{heinrich2003sobolev}}]\label{lem:quantemb}
		Let $D,\tilde D,R,\tilde R$ be nonempty sets and let $F \subset R^D, \tilde F \subset \tilde R^{\tilde D}$ be nonempty and $G$ a normed space. Let further $\tilde S: \tilde F \to G$, $S: F \to G$ and $\Gamma:F \to \tilde F$ be mappings with $S=\tilde S\circ \Gamma$, for which $\Gamma$ is of the following form: There are $\kappa, m^* \in \mathbb N$ and mappings
		\begin{align*}
			\eta_j &:\tilde D \to D \quad (j=0,\dots,\kappa-1)\\
			\beta &: R \to \bbB_{m^*} \\
			\varrho &: \tilde D \times {\bbB_{m^*}}^\kappa \to \tilde R
		\end{align*}
		such that for $f \in F$ and $s \in \tilde D$
		\begin{align}\label{querrymapping}
			\Gamma(f)(s)=\varrho(s,(\beta( f ( \eta_j(s))))_{j=0}^{\kappa-1}).
		\end{align}
		Then for any $n \in \bbN_0$ it holds
		\begin{align*}
			\errquant_{2\kappa n}(S,F)\le \errquant_n(\tilde S, \tilde F).
		\end{align*}
	\end{lem}
	
	\begin{lem}[{\cite[Cor. 2]{heinrich2003sobolev}}]\label{lem:errorofsum<sumoferrors}
		Let $k \in \bbN_0$, $S_\ell:F \to \bbR$ $(\ell=0,\dots,k)$ and $\nu_0,\dots, \nu_k \in \bbN$ such that
		\begin{align*}
			\sum_{\ell=0}^k e^{-\nu_\ell/8}\le \frac{1}{4}.
		\end{align*}
		Then for $n_0,\dots,n_k \in \bbN_0$ and $n\coloneq\sum_{\ell=0}^k \nu_\ell n_\ell$ we have
		\begin{align*}
			\errquant_n\mleft(\sum_{\ell=0}^k S_\ell,F\mright) \le \sum_{\ell=0}^k \errquant_{n_\ell}(S_\ell,F).
		\end{align*}
	\end{lem}
	
	\subsection*{Quantum summation}
	To conclude this section, we will look at an important example in the context of information-based complexity. The summation problem plays a central role both in the construction of our quantum algorithms for integration in Besov classes and in the proof of the corresponding lower bounds. 
	\begin{defi}\label{defi:summation}
		Let $\mathcal N$ be a finite nonempty set and $1 \le p \le \infty$. The space of functions $f:\mathcal N \to \bbR$ equipped with the norm  
		\begin{align*}
			\|f\|_{L_p^{\mathcal N}}\coloneq \begin{cases}
				\mleft(|\mathcal N|^{-1}\sum_{x\in \mathcal N} |f(x)|^p\mright)^{1/p} & \textrm {if } p < \infty\\
				\max\{|f(x)|: x\in \mathcal N\}    &\textrm{if } p = \infty
			\end{cases}
		\end{align*}
		is denoted by $L_p^\mathcal N$. The problem is to obtain an estimate for
		\begin{align*}
			S_\mathcal N f\coloneq \frac{1}{|\mathcal N|}\sum_{x\in \mathcal N} f(x),
		\end{align*}
		with $f \in L_p^\mathcal N$.
	\end{defi}

	\begin{prop}[{\cite[Proposition 1 and Lemma 6]{heinrich2003sobolev}}]\label{prop:quantsum}
		For $1 \le p \le \infty$ there is a constant $c>0$ such that for all $n\in \bbN$ and all finite sets $\mathcal N$ with $2<n\le cN$ ($N=|\mathcal N|$) the following estimates hold
		\begin{enumerate}[label=(\roman*)]
			\item \label{prop:quantsumi} for $1\le p <2$:
			\begin{align*}
				\begin{split}
					\min &\mleft(n^{-2(1-1/p)},n^{-2/p}N^{2/p-1}\mright) 
					\lesssim \errquant_n\mleft(S_\mathcal N,\sB\mleft(L_p^\mathcal N\mright)\mright)\\ 
					&\lesssim \min\mleft(n^{-2(1-1/p)},n^{-2/p}N^{2/p-1}\mright)\max\mleft( \log(n/\sqrt {N}),1 \mright)^{2/p-1},
				\end{split}
			\end{align*}
			\item \label{prop:quantsumii} for $p=2$ :		
			\begin{align*}
				n^{-1} 
				\lesssim \errquant_n\mleft(S_\mathcal N,\sB\mleft(L_2^\mathcal N\mright)\mright)
				\lesssim n^{-1} \lambda(n,N)^{3/2} \log \lambda(n,N)
			\end{align*}
			with $\lambda(n,N) = \log(N/n)+\log\log(n+1)+2$, and
			\item \label{prop:quantsumiii} for $2<p\le \infty$:
			\begin{align*}
				\errquant_n\mleft(S_\mathcal N,\sB\mleft(L_p^\mathcal N\mright)\mright) \asymp n^{-1}.
			\end{align*}
		\end{enumerate}
		However, the restriciton $n \le cN$ is necessary only for the lower bounds.
	\end{prop}
	
	The upper and lower bounds in the case $p=\infty$ were treated by Brassard et al.\ \cite{brassard2000}, as well as by Nayak and Wu \cite{nayak1999}. Heinrich addressed the case $2 < p < \infty$ in \cite{heinrich2002summation}, $p=2$ in \cite{heinrich2003sobolev} and, together with Novak, the case $1 \le p < 2$ in \cite{heinrich2001optimal}.
	
	\section{Periodic Besov spaces with mixed smoothness}
	Here, we give the Fourier-analytic approach to Besov spaces with dominating mixed smoothness on the $d$-dimensional torus based on \cite[Chapter 2]{schmeisser1987}.
	We denote the $d$-torus by $\bbT^d$, represented in $\bbR^d$ by $\bbT^d=[0,1)^d$, where points $x,y \in \bbR^d$ are identified if and only if $x-y \in \bbZ^d$. 
	
	For $0<p \le \infty$ the usual $L_p(\bbT^d)$ spaces are the sets of measurable functions $f:\bbT^d \to \bbC$ such that the quasi norm
	\begin{align*}
		\|f\|_{L_p(\bbT^d)}\coloneq \begin{cases}
			\mleft( \int_{\bbT^d} |f(x)|^p \diffd x \mright)^{1/p} &\textrm{if } p < \infty,\\
			\esssup_{x \in \bbT^d} |f(x)| &\textrm{if } p=\infty,
		\end{cases}
	\end{align*}
	is finite. In case $p\ge 1$ this is in fact a norm. 
	\begin{defi}[Smooth dyadic decomposition of unity]
		Let $\Phi(\bbR)$ be the set of all sequences $(\varphi_n)^{}_{n \in \bbN_0^{}}\subset C_{}^\infty(\bbR)$ such that
		\begin{enumerate}[label=(\roman*)]
			\item $\supp \varphi_n \subset \{x: 2^{n-1} \chi_{\bbN}^{}(n) \le |x| \le 2^{n+1} \} \quad (n \in \bbN_0)$,
			\item $\sup_{x \in \bbR,{n \in \bbN^{}_0}} 2^{n\ell} \mleft| \frac{\diffd^\ell}{\diffd x^\ell}\varphi^{}_n(x)\mright| < \infty \quad (\ell \in \bbN_0)$, and
			\item $\sum_{n=0}^{\infty} \varphi_n(x)=1 \quad (x \in \bbR)$.
		\end{enumerate}
	\end{defi}
	
	\begin{ex}
		As an example of such a sequence we choose a smooth function $\varphi_0$ supported in $[-2,2]$ with  ${\varphi_0}_{|[-1,1]}=1$ and set $\varphi_n(x):=\varphi_{0}(2^{-n}x)-\varphi_{0}(2^{-n+1}x)$ for $n \in \bbN$.
	\end{ex}
	
	Let $(\varphi_n)_{n \in \bbN_0}\in \Phi(\bbR)$ be arbitrary but fixed in the following. The $\bk\in \bbZ^d$-th Fourier coefficient of an integrable $d$-variate periodic function $f$ is given by
	\begin{align*}
		\hat f (\bk) \coloneq \int_{\bbT^d} f(x)e^{-2\pi i \bk \cdot x} \diffd x.
	\end{align*}
	
	For $\bj \in \bbN_0^d$ the $\bj$-th building block $f^\varphi _\bj$ of $f \in L_1(\bbT^d)$ is defined as
	\begin{align*}
		f_\bj^\varphi:\bbT^d \to \bbC, x\mapsto \sum_{\bk \in \bbZ^d}  \hat f (\bk) e^{2\pi i \bk \cdot x}\prod_{i=1}^{d} \mleft( \varphi_{j_i} (k_i) \mright).
	\end{align*}
	
	\begin{defi}
		For $0<p,q\le\infty$ and $r>(1/p-1)_+$ the mixed periodic Besov spaces are the collection of all real valued functions $f\in L_1(\bbT^d)$ such that
		\begin{align*}
			\| f \|_{S_{p,q}^r B(\bbT^d),\varphi} \coloneq \mleft( \sum_{\bj \in \bbN_0^d} 2^{|\bj|_1rq} \|f^\varphi_\bj\|_{L_p(\bbT^d)}^q \mright)^{1/q}<\infty,
		\end{align*}
		with the usual modification $\| f \|_{S_{p,\infty}^r B(\bbT^d),\varphi}\coloneq \sup_{\bj \in \bbN_0^d} (2^{|\bj|_1 r} \|f^\varphi_{\bj}\|_{L_p(\bbT^d)})<\infty $. They are denoted as $S_{p,q}^r B(\bbT^d)$.
	\end{defi}
	
	This definition is independent of the specific choice of $\varphi \in \Phi(\bbR)$ in the sense of equivalent norms or quasi-norms \cite[Sect. 2.2.3]{schmeisser1987}. For simplicity, we will omit the letter $\varphi$ from the notation. Furthermore, the defined spaces are complete \cite[Sect. 2.2.4]{schmeisser1987}. In the present work we are interested in the case $r>1/p$. In this case each element in $S_{p,q}^r B(\bbT^d)$ contains a continuous representative and there is a $c>0$ such that \cite[Lem. 7]{siul07}
	\begin{align}
		\|f\|_{C(\bbT^d)} \le c \| f \|_{\BS}.
		\label{eq:besovemb}
	\end{align}
	
	\	
	\section{The Faber-Schauder basis}
	In this section we construct multilevel quantum algorithms for the approximation of $I_d(f)$, see \eqref{intontorus}. Our construction uses the tensor-product Faber-Schauder decomposition of continuous functions on the $d$-dimensional torus. With this, we break down the problem into a finite number of summation problems.
	
	We start with the definition of the Faber-Schauder system in the univariate case as a foundation for the multivariate decomposition. 
	\begin{defi}
		Let $v:\bbR \to \bbR,x \mapsto(1-|2x-1|)_+$. The Faber-Schauder system is the set of functions
		\begin{align*}
			\{ v_{j,k}: j \in \bbN_{-1}, k \in \bbB_j \},
		\end{align*}
		where $v_{j,k}:\bbT \to \bbR$ are defined as follows:
		\begin{align*}
			v_{-1,0}(x)&\coloneqq 1\quad\textrm{and}\quad
			v_{j,k}(x) \coloneqq v(2^jx-k) \textrm{ for } j\in \bbN_0, k\in \bbB_j.
		\end{align*}
		\label{defi:faber-basis}
	\end{defi}
	
	\begin{figure}
		\centering
		\begin{subfigure}{\textwidth}
			\centering
			\begin{tikzpicture}[scale=2.4]
				\begin{scope}[shift={(0,0)}]
					\achsen;
					\draw[domain=0:1, samples=300] plot (\x, {max(1-abs(2*\x-1),0)});
					\node[above,font=\small,inner sep=2pt] at (1/2,1) {$v_{0,0}$};
				\end{scope}
				
				\begin{scope}[shift={(1.4,0)}]
					\achsen;
					\draw[domain=0:1, samples=300] plot (\x, {max(1-abs(2*2*\x-1),0)});
					\node[above,font=\small,inner sep=2pt] at (1/4,1) {$v_{1,0}$};
				\end{scope}
				
				\begin{scope}[shift={(2.8,0)}]
					\achsen;
					\draw[domain=0:1, samples=300] plot (\x, {max(1-abs(2*2*\x-1-2),0)});
					\node[above,font=\small,inner sep=2pt] at (3/4,1) {$v_{1,1}$};
				\end{scope}
				
				\begin{scope}[shift={(4.2,0)}]
					\achsen;
					\draw[dashed,domain=0:1, samples=300] plot (\x, {max(1-abs(2*\x-1),0)});
					\node[above,font=\small,inner sep=2pt] at (1/2,1) {$v_{0,0}$};
					\draw[domain=0:1, samples=300] plot (\x, {max(1-abs(2*2*\x-1),0)});
					\node[above,font=\small,inner sep=2pt] at (1/4,1) {$v_{1,0}$};
					\draw[domain=0:1, samples=300] plot (\x, {max(1-abs(2*2*\x-1-2),0)});
					\node[above,font=\small,inner sep=2pt] at (3/4,1) {$v_{1,1}$};
				\end{scope}
			\end{tikzpicture}
			\caption{Hierarchical Faber functions for the levels $j=0,1$ and their union. Upon transitioning from level $j \ge 0$ to $j+1$, the number of hat functions doubles, whereas the area beneath each function is reduced by half.}
			\label{fig:faberhats}
		\end{subfigure}
		\par\bigskip
		\begin{subfigure}{\textwidth}
			\centering
			\begin{tikzpicture}[yscale=2.5,xscale=2.7,declare function={f(\t)=0.7*sin((1.4*\t) r )+0.1;},declare function={g(\t)=f(\t);}]
				\begin{scope}[shift={(0,0)}]
					\draw[-{Stealth}] (-0.14,0) -- (2.15,0);
					\draw (0,-0.015) -- (0,0.015);
					\draw (1,-0.015) -- (1,0.015);
					\draw (2,-0.015) -- (2,0.015);
					\node[below,font=\small] at (0,0) {$2^{-j}k$};
					\node[below,font=\small] at (1,0) {$2^{-j}(k+1/2)$};
					\node[below,font=\small] at (2,-0.015) {$2^{-j}(k+1)$};
					
					\draw[domain=-0.05:2.05, samples=300] plot (\x, {g(\x)} );
					\node[above right] at (1.6,0.6) {$f(x)$};
					
					\draw [line cap=round,dash pattern=on 0pt off 1pt,gray] (0,0) -- (0,{g(0)});
					\draw [line cap=round,dash pattern=on 0pt off 1pt,gray] (1,0) -- (1,{g(1)});
					\draw [line cap=round,dash pattern=on 0pt off 1pt,gray] (2,0) -- (2,{g(2)});
					\draw[dashed] (0,{g(0)}) -- (2,{g(2)});
					
					\draw[decorate,decoration={brace,mirror,amplitude=4pt}] (1, {(g(2)+g(0))/2}  ) -- (1,{g(1)}) node[midway,right=2pt,font=\small]{$d_{j,k}(f)$};
				\end{scope}
				
				\begin{scope}[shift={(2.6,0)}]
					\draw[-{Stealth}] (-0.14,0) -- (2.15,0);
					\draw (0,-0.015) -- (0,0.015);
					\draw (1/2,-0.015) -- (1/2,0.015);
					\draw (1,-0.015) -- (1,0.015);
					\draw (3/2,-0.015) -- (3/2,0.015);
					\draw (2,-0.015) -- (2,0.015);
					
					\draw[domain=-0.05:2.05, samples=300] plot (\x, {g(\x)} );
					\node[above right] at (1.6,0.6) {$f(x)$};
					
					\draw[line cap=round,dash pattern=on 0pt off 1pt] (0,{g(0)})  -- (1/2,{g(1/2)}) -- (1,{g(1)}) -- (3/2,{g(3/2)}) -- (2,{g(2)}) ;
					\draw[line cap=round,dash pattern=on 1pt off 1pt] (0,{g(0)}) -- (1,{g(1)}) -- (2,{g(2)}) ;
					\draw[line cap=round,dash pattern=on 2pt off 1pt] (0,{g(0)})  -- (2,{g(2)}) ;
					
					\node[below,font=\small] at (0,0) {$2^{-j}k$};
					\node[below,font=\small] at (1,0) {$2^{-j}(k+1/2)$};
					\node[below,font=\small] at (2,-0.015) {$2^{-j}(k+1)$};
				\end{scope}
			\end{tikzpicture}
			\caption{Left: a Faber coefficient as difference between the function value at the midpoint and the mean of the endpoint values. Right: Piecewise linear interpolants on three nested dyadic grids.}
			\label{fig:fabercoeff}
		\end{subfigure}
		\caption{Illustration of the univariate Faber expansion}
		\label{fig:gesamt}
	\end{figure}
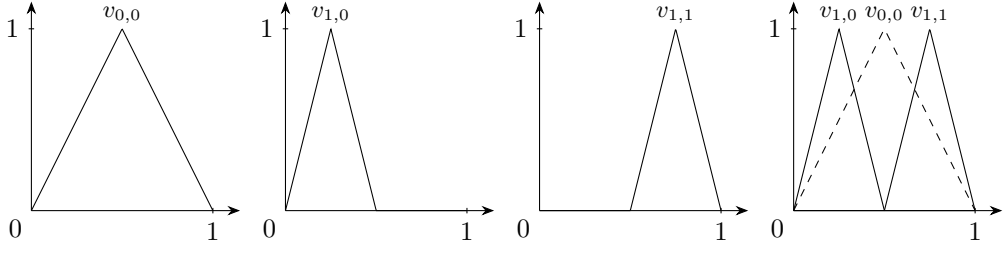
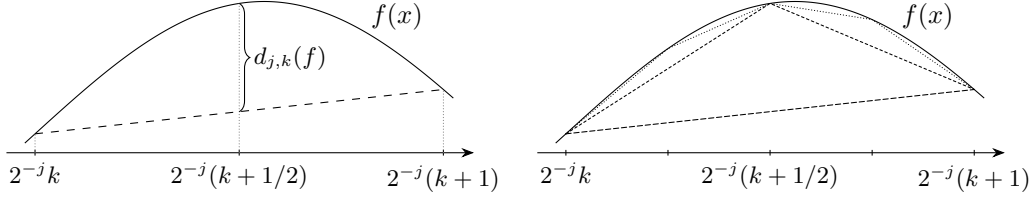
	
	These functions build a piecewise linear hierarchical basis \cite{bungartz2004}. The first non-negative levels of the Faber system are illustrated in Figure \ref{fig:faberhats}. 
	
	Faber \cite{faber1908} pointed out that every continuous function $f$ on $\bbT$ has a uniform series expansion in the following sense
	\begin{align}\label{eq:faberunivar}
		f=\sum_{j \in \bbN_{-1}}\sum_{k\in \bbB_j} d_{j,k}(f) v_{j,k}
	\end{align}
	with coefficients 
	\begin{align*}
		d_{j,k}(f)\coloneqq \begin{cases}
			f(0) \quad &\textrm{if } j=-1,\\
			-\frac{1}{2}\Delta_{2^{-j-1}}^2f(2^{-j}k) & \textrm{if } j \in \bbN_0,
		\end{cases}
	\end{align*}
	where $\Delta^2_h f (x)$ denotes the second difference of $f$ at $x$ with step size $h$, defined by
	\begin{align*}
		\Delta_h^2f(x) = f(x+2h)-2f(x+h)+f(x).
	\end{align*}
	An illustration of the Faber coefficients $d_{j,k}(f)$ is given in Figure \ref{fig:fabercoeff}.
	
	For multivariate functions $f \in C(\bbT^d)$ by fixing all variables except one and successively applying \eqref{eq:faberunivar} we get the multivariate Faber expansion 
	\begin{align}\label{eq:fabermultivar} 
		f=\sum_{\bj \in \bbN_{-1}^d}\sum_{\bk\in \bbB_\bj} d_{\bj,\bk}(f) v_{\bj,\bk}.
	\end{align}
	Here $\bbB_\bj$ denotes the product $\bbB_{j_1} \times \dots \times \bbB_{j_d} $,
	\begin{align}
		v_{\bj,\bk}(x) \coloneqq \prod_{i=1}^d v_{j_i,k_i}(x_i)
		\label{eq:multifaber}
	\end{align}
	and 
	\begin{align}\label{eq:defi:fabercoeff}
		d_{\bj,\bk}(f) \coloneqq \mleft(-\frac{1}{2}\mright)^{|\{i:j_i\ge 0\}|}\sum_{\bl \in U_\bj} (-2)^{|\{i:\ell_i=1\}|} f(x_{\bj,\bk,\bl}),
	\end{align}
	whereby $U_\bj \coloneq U_{j_1} \times \dots \times U_{j_d}$ with 
	\begin{align}\label{eq:defi:Uj}
		U_{j_i} \coloneq\begin{cases}
			\{0\} &\textrm{if } j_i = -1,\\
			\{0,1,2\} &\textrm{if } j_i \in \bbN_0.
		\end{cases}
	\end{align}
	An illustration of the tensorised Faber functions is given in \cite{bungartz2004}. The function $f$ is sampled at the points
	\begin{align}\label{eq:samplepoints}
		x_{\bj,\bk,\bl} \coloneqq \mleft( 2^{-(j_1)_+}(k_1+\ell_1/2),\dots,2^{-(j_d)_+}(k_d+\ell_d/2) \mright).
	\end{align}
	
	\begin{defi}
		For $0<p,q\le \infty$ and $r\in \bbR$ let $s_{p,q}^rb$ be the space of all real-valued sequences $(\lambda_{\bj,\bk})_{\bj\in \bbN_{-1}^d, \bk \in \bbB_\bj}$ such that
		\begin{align*}
			\| \lambda_{\bj,\bk} \|_{s_{p,q}^rb} \coloneq \mleft[ \sum_{\bj \in \bbN_{-1}^d} 2^{|\bj|_1(r-1/p)q}\mleft( \sum_{\bk \in \bbB_{\bj}} |\lambda_{\bj,\bk}|^p  \mright)^{q/p} \mright]^{1/q}<\infty,
		\end{align*}
		with the usual modifications in the cases where $p=\infty$ or $q=\infty$.
	\end{defi}

	\begin{prop}\cite[Proposition 3.4]{hinrichs2016}\label{prop:SBinsb}
		For $1/2<p\le \infty, 0<q\le \infty$ and $1/p<r<2$ and $f\in C(\bbT^d)$ the following holds
		\begin{align*}
			\|d_{\bj,\bk}(f)\|_{s^r_{p,q}b} \lesssim \|f\|_{\BS}.
		\end{align*}
	\end{prop}
	
	\section{Quantum algorithms for numerical integration} 
	
	\begin{proof}[Proof of Theorem \ref{thm:besovintup}]
		Let $1\le p,q < \infty$ (The cases $p=\infty$ or $q=\infty$ are treated analogously using the usual modifications) and $\max\{1/p,2/q-1\}<r<2$.
		
		\smallskip
		\noindent\textit{Step 1:} \textbf{Preliminary work and decomposition of the integral.} We set $s \coloneq \min\{p,q\}$ and for $\ell,L \in \bbN_0$ we set 
		\begin{align*}
			\mcV_\ell \coloneq\mleft\{\bj \in \bbN^d_{-1}:\sum_{i=1}^d j_i = \ell-d\mright\}, \quad \mcW_\ell \coloneq \{ (\bj,\bk): \bj \in \mcV_\ell, \bk\in \bbB_\bj \}
		\end{align*}
		and 
		\begin{align}\label{eq:defi:J}
			J_L f \coloneq \sum_{\ell=0}^L 2^{-\ell} \sum_{(\bj,\bk)\in \mcW_\ell} d_{\bj,\bk}(f),
		\end{align}
		where $d_{\bj,\bk}(f)$ are the Faber coefficients in \eqref{eq:fabermultivar}. We note that
		\begin{align}\label{eq:defWl}
			W_\ell \coloneq|\mcW_\ell| \asymp 2^\ell|\mcV_\ell|\asymp 2^\ell \ell^{d-1}.
		\end{align}
		Further let $F$ and $\tilde F_\ell$ denote the closed unit balls of $\BS$ and $L_{s}^{\mcW_\ell}$ respectively.
		
		Now let us fix some parameters for the algorithms. Let $\delta>0$ such that
		\begin{align}\label{eq:defi:delta}
			\delta < \min \mleft( r,\frac{s(1+r)}{2}-1 \mright),
		\end{align}
		and set $K\coloneq 1+2\sum_{\ell=1}^\infty \lceil 8 (2 \ln \ell+\ln 8)\rceil 2^{-\ell\delta}$.

		Without loss of generality we assume $n \in \bbN$ with $n \ge 6(6d)^{d-1}K$ and set
		\begin{align}
			\lowl &\coloneq\max\{\ell \in \bbN: 3^d  \ell^{d-1}2^\ell K \le n\}, \label{defi:lowl}\\
			\highl &\coloneq\min\{\ell \in \bbN: \ell^{(d-1)(1-1/q)}2^{-r\ell}\le \lowl^{-(d-1)/q}2^{-(r+1)\lowl}\}, \label{eq:defi:highl}
		\end{align}
		and
		\begin{align}\label{defi:ntilde}
			\tilde n \coloneq 3^d\lowl ^{d-1}2^\lowl.
		\end{align}
		Using the approximation $J_L$ defined in \eqref{eq:defi:J}, we decompose the integral into the sum
		\begin{align*}
			I_d f= (I_df-J_\highl f)+J_\lowl f+(J_\highl f- J_\lowl f),
		\end{align*}
		and note that $n \ge K \tilde n > \tilde n$.
		So we can apply Lemma 6 (i) from \cite{heinrich2002summation} and Lemma 2 from \cite{heinrich2003sobolev} and obtain the following estimate
		\begin{align}\label{mainerrorest}
			\errquant_n(I_d,F)\le \sup_{f \in F}|I_df-J_\highl f|+\errquant_{\tilde n}(J_\lowl,F,0)+\errquant_{n-\tilde n}(J_\highl -J_\lowl ,F).
		\end{align}
		In what follows, we evaluate the three terms separately.
		
		\smallskip
		\noindent\textit{Step 2.} \textbf{Estimation of the truncation error.} For $f \in F$ the multivariate Faber expansion \eqref{eq:fabermultivar} converges uniformly \cite{hinrichs2016}. Together with \eqref{eq:multifaber} and Definition \ref{defi:faber-basis} this yields
		\begin{align*}
			|I_df&-J_\highl f| 
			\le \sum_{\ell=\highl+1}^{\infty}\sum_{(\bj,\bk)\in \mcW_\ell}  | d_{\bj,\bk}(f) | \int_{[0,1)^d} v_{\bj,\bk}(x) \diffd x \\
			&\lesssim \sum_{\ell=\highl+1}^{\infty}\sum_{\bj \in \mcV_\ell} 2^{-|\bj|_1/p} \mleft(\sum_{\bk\in \bbB_\bj} | d_{\bj,\bk}(f) |^p\mright)^{1/p} \\
			&\le \mleft( \sum_{\ell=\highl+1}^{\infty}\sum_{\bj \in \mcV_\ell} 2^{-r|\bj|_1 q'} \mright)^{1/q'} \mleft[ \sum_{\bj \in \bbN_{-1}^d}  2^{(r-1/p)|\bj|_1q} \mleft(\sum_{\bk\in \bbB_\bj} | d_{\bj,\bk}(f) |^p\mright)^{q/p} \mright]^{1/q}\\
			&\lesssim\mleft( \sum_{\ell=\highl+1}^\infty |\mcV_\ell|2^{-r\ell q'}   \mright)^{1/q'}\|d_{\bj,\bk}(f)\|_{s^r_{p,q}b} \\
			&\lesssim \mleft(\sum_{\ell=\highl+1}^\infty  \ell^{d-1} 2^{-r\ell q'}\mright)^{1/q'}\|d_{\bj,\bk}(f)\|_{s^r_{p,q}b} \\
			&\lesssim \highl^{(d-1)(1-1/q)}2^{-r\highl}\|d_{\bj,\bk}(f)\|_{s^r_{p,q}b},\\
		\end{align*}
		whereby Hölder's inequality is used in the second estimate for the sum over $\bbB_\bj$ with exponent $p$ and in the third estimate with the conjugate exponents $q$ and $q'$ ($1/q+1/q'=1$). 
		
		Applying Proposition \ref{prop:SBinsb}, \eqref{eq:defi:highl} and \eqref{defi:lowl} we thus obtain the following estimate for the first term in equation \eqref{mainerrorest}
		\begin{align}
			\sup_{f \in F}|I_df-J_\highl f| 
			&\lesssim \highl^{(d-1)(1-1/q)}2^{-r\highl}
			\lesssim \lowl^{-(d-1)/q}2^{-(r+1)\lowl}\notag\\
			&= (2^\lowl \lowl^{d-1})^{-(r+1)} \lowl^{(d-1)(r+1-1/q)} \notag\\
			&\lesssim n^{-r-1}(\log n)^{(d-1)(r+1-1/q)}. \label{errortrunc}
		\end{align}
		
		\smallskip
		\noindent\textit{Step 3.} \textbf{Evaluation of the lower levels.} By \eqref{eq:defi:fabercoeff} and \eqref{eq:defi:Uj} every coefficient $d_{\bj,\bk}(f)$ depends on at most $3^d$ sampling points. Therefore the number of samples entering $J_\lowl f$ can be estimated using \eqref{defi:ntilde} by 
		\begin{align*}
			\sum_{\ell=0}^{\lowl} \sum_{\bj \in \mcV_\ell} |\bbB_\bj|| U_\bj|
			&\le \frac{3^d}{2^d} \sum_{\ell=0}^{\lowl} 2^\ell |\mcV_\ell| 
			\le \frac{3^d}{2^d (d-1)!} \sum_{\ell=0}^{\lowl} (\ell+d-1)^{d-1}2^\ell\\
			&\le  3^d\frac{\lowl^{d-1}}{(d-1)!}  2^{\lowl}
			\le 3^d \lowl^{d-1}2^{\lowl}\le \tilde n.
		\end{align*}
		This yields 
		\begin{align}\label{errorlowerlevels}
			\errquant_{\tilde n}(J_{\lowl},F,0)=0,
		\end{align}
		so $J_{\lowl}$ can be computed to arbitrary precision using sufficiently many qubits.				
	
		\smallskip
		\noindent\textit{Step 4.} \textbf{Reduction of the intermediate levels to quantum summation.}
		This is the technical main part of the proof. To estimate the last term in equation \eqref{mainerrorest}
		we decompose $J_\highl  -J_\lowl$ into its levels
		\begin{align}\label{eq:tele}
			J_\highl  -J_\lowl = \sum_{\ell=\lowl+1}^{\highl} J'_\ell,
		\end{align}
		where, by \eqref{eq:defi:J}
		\begin{align}\label{eq:jstroke}
			J'_\ell f \coloneq J_\ell f-J_{\ell-1}f=2^{-\ell} \sum_{(\bj,\bk)\in \mcW_\ell} d_{\bj,\bk}(f).
		\end{align}
		For $\ell \in \{\lowl+1,\cdots,\highl\}$ we set
		\begin{align}\label{eq:defnuldefnl}
			\nu_\ell \coloneq \lceil 8 (2 \ln(\ell-\lowl)+\ln 8)\rceil, \quad
			n_\ell \coloneq 
			2^\lowl \lowl^{d-1}2^{-(\ell-\lowl)\delta}
		\end{align}
		to get
		\begin{align*}
			\sum_{\ell=\lowl+1}^{\highl} e^{-\nu_\ell/8} 
			\le \frac{1}{8}\sum_{\ell=1}^{\infty}\frac{1}{\ell^2} 
			< \frac{1}{4}
		\end{align*}
		and with \eqref{defi:lowl}
		\begin{align*}
			2\cdot 3^d\sum_{\ell=\lowl+1}^{\highl} \nu_\ell \lfloor n_\ell \rfloor 
			&\le 3^d \lowl^{d-1} 2^\lowl  2\sum_{\ell=1}^{\infty} \lceil 8 (2 \ln \ell+\ln 8)\rceil 2^{-\ell\delta}\notag \\
			&\le \tilde n (K-1)
			\le n-\tilde n.
		\end{align*}
		
		By using \eqref{eq:tele} and applying Lemma \ref{lem:errorofsum<sumoferrors} we see 
		\begin{align}\label{eq:errsumJl'}
			\errquant_{n-\tilde n}(J_\highl  -J_\lowl,F)
			\le \sum_{\ell=\lowl+1}^{\highl} \errquant_{2\cdot3^d\lfloor n_\ell \rfloor}( J'_{\ell},F).
		\end{align}
		
		At this point we define mappings $\Gamma_\ell: F \to L_s^{\mcW_\ell}, \ell=\lowl,\dots,\highl$ which allow us to reduce the problem (in the sense of Lemma \ref{lem:quantemb}) to the computation of means in $L_s^{\mcW_\ell}$. For this, let $m^*\in \bbN$ such that 
		\begin{align}\label{eq:defm*}
			2^{-m^*/2}
			&\le 2^{-r\highl} \highl^{-d}\quad \text{and} \quad 
			c \le 2^{m^*/2-1},
		\end{align}
		where $c$ is the constant from \eqref{eq:besovemb}. We first define mappings $\beta, \gamma$ and $\eta_{\ell,\bl}, \varrho_\ell$ for ($\bl \in \{0,1,2\}^d$)  by
		\begin{align}
			\begin{split}\label{eq:def_beta_gamma}
				\beta &: \bbR \to \bbB_{m^*}, z \mapsto \min\mleft( (\lfloor 2^{m^*/2}z+2^{m^*-1} \rfloor)_+,2^{m^*}-1 \mright),\\
				\gamma &: \bbB_{m^*} \to \bbR, y \mapsto 2^{-m^*/2}y-2^{m^*/2-1},
			\end{split}\\
			\eta_{\ell,\bl}&: \mcW_\ell \to \bbT^d, (\bj,\bk) \mapsto 
			\begin{cases}
				x_{\bj,\bk,\bl} &\textrm{if } \bl \in U_\bj \\
				0 &\textrm{otherwise},
			\end{cases} \nonumber \\
			\varrho_\ell&: \mcW_\ell \times \bbB_{m^*}^{\{0,1,2\}^d} \to \bbR, ((\bj,\bk),y) \mapsto \mleft(-\frac{1}{2}\mright)^{|\{i:j_i\ge 0\}|}\frac{W_\ell}{2^{\ell}}\sum_{\bl \in U_\bj}(-2)^{|\{i:\ell_i=1\}|}\gamma(y_{\bl}), \nonumber
		\end{align}
		where $x_{\bj,\bk,\bl}$ are defined in (\ref{eq:samplepoints}). Finally we set
		\begin{align*}
			\Gamma_\ell(f)(\bj,\bk)\coloneq \varrho_\ell((\bj,\bk),(\beta\circ f\circ \eta_{\ell,\bl}(\bj,\bk))_{\bl \in \{0,1,2\}^d}).
		\end{align*}
		
		For $|z| \le 2^{m^*/2-1}$ we have
		\begin{align}\label{distretizationerror}
			|z-\gamma\circ \beta(z)|\le 2^{-m^*/2},
		\end{align}
		which, due to \eqref{eq:defm*}, applies in particular for $f\in F$. Again with \eqref{eq:defm*} and \eqref{eq:defi:fabercoeff} we get
		\begin{align}
			\label{eq:diffdgamma}
			\mleft|d_{\bj,\bk}(f)W_\ell2^{-\ell}-\Gamma_\ell(f)(\bj,\bk)\mright|
			&\le \frac{W_\ell}{2^{\ell}} \sum_{\bl \in U_\bj} 2^{|\{i:\ell_i=1\}|}| f(x_{\bj,\bk,\bl})-\gamma\circ \beta \circ f(x_{\bj,\bk,\bl})| \notag\\ 
			&\lesssim \ell^{d-1} 2^{-m^*/2}
			\le \highl^{-1} 2^{-r\highl},
		\end{align}
		and therefore \eqref{eq:jstroke} yields
		\begin{align*}
			| J_\ell'f-S_{\mcW_\ell} \Gamma_\ell(f) | 
			&\le \frac{1}{W_\ell} \sum_{(\bj,\bk)\in \mcW_\ell} |d_{\bj,\bk}(f)W_\ell 2^{-\ell}-\Gamma_\ell(f)(\bj,\bk) | 
			\lesssim \highl^{-1} 2^{-r\highl},
		\end{align*}
		where $S_{\mcW_\ell}$ is the summation operator, see Definition \ref{defi:summation}. Using \cite[Lem. 6 (i)]{heinrich2002summation}, we thus obtain
		\begin{align}\label{eq:errJl'}
			\errquant_{2\cdot 3^d\lfloor n_\ell\rfloor}( J'_{\ell},F) & \le
			\errquant_{2\cdot 3^d\lfloor n_\ell \rfloor}( S_{\mcW_\ell} \Gamma_\ell,F) + \highl^{-1} 2^{-r\highl}.
		\end{align}
		
		Now we estimate the norm of $\Gamma_\ell(f)$ in $L^{\mcW_\ell}_{s}$ for $f \in F$. By using Hölder's inequality (regarding the outer sum in the case where $s = p$ and the inner sum in the other case) and Proposition \ref{prop:SBinsb} we get
		\begin{align*}
			\| d_{\bj,\bk}(f)\|_{L^{\mcW_\ell}_{s}} 
			&=\mleft(  \frac{1}{W_\ell}\sum_{\bj\in \mcV_\ell} \sum_{\bk \in \bbB_\bj}|d_{\bj,\bk}(f)|^{s} \mright)^{1/s}\\
			&\lesssim 2^{-r\ell}\ell^{-(d-1)/q}\mleft(\sum_{\bj\in \mcV_\ell} 2^{|\bj|_1(r-1/p)q}\mleft( \sum_{\bk \in \bbB_\bj}|d_{\bj,\bk}(f)|^p \mright)^{q/p}\mright)^{1/q}\\
			&\lesssim 2^{-r\ell}\ell^{-(d-1)/q} \|d_{\bj,\bk}(f)\|_{s^r_{p,q}b}\\
			&\lesssim 2^{-r\ell}\ell^{-(d-1)/q}.
		\end{align*}
		
		In combination with \eqref{eq:diffdgamma},\eqref{eq:defWl} and \eqref{eq:defm*} this results in
		\begin{align*}
			\|\Gamma_\ell(f)\|_{L^{\mcW_\ell}_{s}}
			&\le   W_\ell 2^{-\ell} \| d_{\bj,\bk}(f)\|_{L^{\mcW_\ell}_{s}}
			+ \| \Gamma_\ell(f)(\bj,\bk) - d_{\bj,\bk}(f)W_\ell2^{-\ell}   \|_{L_\infty^{\mcW_\ell}}\\
			&\lesssim 2^{-r\ell}\ell^{(d-1)(1-1/q)} + \ell^{-1}2^{-r\ell}\\
			&\lesssim 2^{-r\ell} \ell^{(d-1)(1-1/q)},
		\end{align*}
		i.e. there is a $c>0$ such that
		\begin{align*}
			\Gamma_\ell(F) \subseteq c 2^{-r\ell}\ell^{(d-1)(1-1/q)} \tilde F_\ell.
		\end{align*}
		
		An application of Lemma \ref{lem:quantemb} yields
		\begin{align*}
			\errquant_{2\cdot 3^d\lfloor n_\ell \rfloor}( S_{\mcW_\ell} \Gamma_\ell,F)
			\lesssim 2^{-r\ell}\ell^{(d-1)(1-1/q)} \errquant_{\lfloor n_\ell \rfloor}( S_{\mcW_\ell} , \tilde F_\ell )
		\end{align*}
		and together with \eqref{eq:errsumJl'} and \eqref{eq:errJl'} we obtain
		\begin{align*}
			\errquant_{n-\tilde n}(J_\highl-J_\lowl,F) 
			&\lesssim  2^{-r\highl}+\sum_{\ell=\lowl+1}^{\highl} 2^{-r\ell}\ell^{(d-1)(1-1/q)} \errquant_{\lfloor n_\ell \rfloor }( S_{\mcW_\ell} ,\tilde F_\ell).
		\end{align*}
		With \eqref{eq:defi:highl} we estimate the first term on the right hand side to
		\begin{align*}
			2^{-r\highl}
			&\le \highl^{(d-1)(1-1/q)}2^{-r\highl} 
			\le \lowl^{-(d-1)/q}2^{-(r+1)\lowl}\notag \\
			&\lesssim n^{-r-1}(\log n)^{(d-1)(r+1-1/q)}
		\end{align*}
		and get
		\begin{align}\label{eq:errIdfin}
			\errquant_{n-\tilde n}(J_\highl-J_\lowl,F) 
			&\lesssim  n^{-r-1}(\log n)^{(d-1)(r+1-1/q)}+E(n).
		\end{align}
		with 
		\begin{align*}
			E(n)\coloneq \sum_{\ell=\lowl+1}^{\highl} 2^{-r\ell}\ell^{(d-1)(1-1/q)} \errquant_{\lfloor n_\ell \rfloor }( S_{\mcW_\ell} ,\tilde F_\ell)
		\end{align*}
		
		\smallskip
		\noindent\textit{Step 5.} \textbf{Estimation of the summation error.} It remains to find estimates for $E(n)$ in different parameter regimes. We have to distinguish between three cases. In each case the choice of $\delta$ in \eqref{eq:defi:delta} ensures sufficient geometric decay to control the corresponding sum and \eqref{defi:lowl} implies $\lowl^{d-1}2^\lowl \asymp n$. First we consider the case $s>2$. From Proposition \ref{prop:quantsum}\ref{prop:quantsumiii} (if $n_\ell<3$ we use the trivial error bounds) and \eqref{eq:defnuldefnl} we conclude
		\begin{align*}
			\errquant_{\lfloor n_\ell \rfloor }( S_{\mcW_\ell} , \tilde F_\ell ) \lesssim n_\ell^{-1}\le \lowl^{-(d-1)}2^{-\lowl}2^{(\ell-\lowl)\delta}.
		\end{align*}
		By using this inequality in \eqref{eq:errIdfin} and taking \eqref{defi:lowl} into account we get
		\begin{align*}
			E(n) 
			&\lesssim \lowl^{-(d-1)}2^{-(r+1)\lowl}\sum_{\ell=\lowl+1}^{\highl} \ell^{(d-1)(1-1/q)} 2^{-(r-\delta)(\ell-\lowl)}\\
			&\lesssim 2^{-(r+1)\lowl} \lowl^{-(d-1)/q}\\
			&\lesssim (2^{\lowl}\lowl^{d-1})^{-(1+r)}  (\lowl+\log \lowl)^{(d-1)(r+1-1/q)}  \\
			&\lesssim n^{-r-1} (\log n)^{(d-1)(r+1-1/q)}.
		\end{align*}
		
		Next we prove the upper bound in the case $1\le s < 2$. From Proposition \ref{prop:quantsum}\ref{prop:quantsumi}, \eqref{eq:defnuldefnl} and \eqref{eq:defWl} we get
		\begin{align*}
			\errquant_{\lfloor n_\ell \rfloor }( S_{\mcW_\ell} ,\tilde F_\ell)
			&\lesssim \min\{n_\ell^{-2(1-1/s)},n_\ell^{-2/s}W_\ell^{2/s-1}\}\max\{\log(n_\ell/\sqrt {W_\ell}),1\}^{2/s-1}\\
			& \lesssim n_\ell^{-2/s}W_\ell^{2/s-1} (1+\max\{\log(n_\ell/\sqrt {W_\ell}),0\})^{2/s-1} \\
			& \lesssim n_\ell^{-2/s}W_\ell^{2/s-1} (1+ \lowl + \log \lowl)^{2/s-1} \\
			&\lesssim \lowl^{-2(d-1)/s}\ell^{(d-1)(2/s-1)} 2^{-\lowl} 2^{-(\ell-\lowl)(s-2-2\delta)/s}(\log\lowl+\lowl)^{2/s-1}.
		\end{align*}
		Together with \eqref{eq:errIdfin}, and \eqref{defi:lowl} this results in the following estimate
		\begin{align*}
			E(n)
			& \lesssim 2^{-(r+1)\lowl}\lowl^{-2(d-1)/s}(\log \lowl + \lowl)^{2/s-1} \\
			& \phantom{==} \cdot\!  \sum_{\ell=\lowl+1}^{\highl} 2^{-(\ell-\lowl)(s+rs-2-2\delta)/s} \ell^{(d-1)(2/s-1/q)}\\
			&\lesssim 2^{-(r+1)\lowl} \lowl^{-(d-1)/q}(\log \lowl + \lowl)^{2/s-1} \\
			&\lesssim (2^\lowl \lowl^{d-1})^{-(r+1)} (\lowl+\log \lowl)^{(d-1)(r+1-1/q)+2/s-1}\\
			&\lesssim n^{-r-1} (\log n)^{(d-1)(r+1-1/q)+2/s-1}\,.  
		\end{align*}
		
		Finally we consider the case $s=2$. An application of Proposition \ref{prop:quantsum}\ref{prop:quantsumii} together with \eqref{eq:defnuldefnl} and \eqref{eq:defWl} yields
		\begin{align*}
			\errquant_{\lfloor n_\ell \rfloor }( S_{\mcW_\ell} ,\tilde F_\ell)
			&\lesssim n_\ell^{-1} (\log(W_\ell/n_\ell)+\log\log(n_\ell+1)+2)^{3/2} \\
			&\phantom{==}\cdot \log(\log(W_\ell/n_\ell)+\log\log(n_\ell+1)+2)\\
			&\lesssim \lowl^{-(d-1)}2^{-\lowl}2^{(\ell-\lowl)\delta}(\ell-\lowl+\log(\ell/\lowl)+\log\log n+1)^{3/2}\\
			&\phantom{==}\cdot \log(\ell-\lowl+\log(\ell/\lowl)+\log\log n+1)\,.
		\end{align*}
		We plug this into \eqref{eq:errIdfin}, use again \eqref{defi:lowl} and get the upper bound of the last case:
		\begin{align*}
			E(n)
			&\lesssim 2^{-(1+r)\lowl}\lowl^{-(d-1)} \sum_{\ell=\lowl+1}^{\highl}2^{-(r-\delta)(\ell-\lowl)}\ell^{(d-1)(1-1/q)}\\
			& \phantom{=============} \cdot (\ell-\lowl+\log(\ell/\lowl)+\log\log n+1)^{3/2} \\
			& \phantom{=============} \cdot \log(\ell-\lowl+\log(\ell/\lowl)+\log\log n)\\
			&\lesssim 2^{-(1+r)\lowl} \lowl^{-(d-1)/q}(\log\log n)^{3/2} \log\log\log n\\
			&\lesssim (2^{\lowl}\lowl^{d-1})^{-1-r}(\lowl+\log\lowl)^{(d-1)(r+1-1/q)}(\log\log n)^{3/2} \log\log\log n\\
			&\lesssim n^{-1-r}(\log n)^{(d-1)(r+1-1/q)}(\log\log n)^{3/2} \log\log\log n.
		\end{align*}
		In summary, \eqref{mainerrorest}, \eqref{errortrunc}, \eqref{errorlowerlevels}, \eqref{eq:errIdfin}, and the estimates of $E(n)$ in the three different cases yield the claimed rates.
	\end{proof}
	
	\section{Lower bounds}
	In this section we want to prove the lower bound of the integration problem. Our proof uses an embedding result for periodic Besov spaces based on atomic decompositions from \cite[Proposition 5.5.10.]{dung2018} which is based on \cite[Thm. 2.4]{vybiral2006}.
	\begin{prop}\label{prop:normatomdecomb}
		Let $1\le p,q\le \infty$, $r>0$, $\varphi \in C^\infty(\bbR)$ with $\supp(\varphi) \subset (0,1)$ and
		\begin{align*}
			a_{\bj,\bk}(x)
			\coloneq \prod_{i=1}^{d}\varphi(2^{j_i}x_i-k_i),    \quad \bj\in \bbN_{0}^d, \bk \in \bbB_\bj, x \in \bbR^d.
		\end{align*}
		Further let $(\lambda_{\bj,\bk})_{\bj\in \bbN_{0}^d, \bk \in \bbB_\bj} \in s_{p,q}^rb$. Then 
		\begin{align*}
			f = \sum_{\bj\in \bbN_{0}^d }\sum_{\bk \in \bbB_\bj} \lambda_{\bj,\bk}a_{\bj,\bk}
		\end{align*}
		belongs to $\BS$ and 
		\begin{align*}
			\| f \|_{\BS} 
			&\lesssim \| \lambda \|_{s_{p,q}^rb},
		\end{align*}
		where the constant depends on $r,p,q$ and $\varphi$.
	\end{prop}
	
	\subsection{Construction of test functions}
	In the classical setting, it is common to construct test functions for establishing lower bounds. A similar approach is adopted in the quantum context:  We formulate a sufficiently hard summation problem and, using Lemma \ref{lem:quantemb}, reduce it to the integration problem in $\BS$. The image of the reduction mapping $\Gamma$ serves as our test function class. The construction of $\Gamma$ draws on classical constructions of test functions in spaces with dominating mixed smoothness. A precursor to such a construction can be found in Triebel \cite[Proof of Thrm. 4.15, step 3]{triebel2010}; later generalizations were provided by D\~ung and T. Ullrich \cite[Sect. 4]{du2015} and M. and T. Ullrich \cite[Sect. 7.2]{uu2016}. To ensure that $\Gamma$ has the form given in \eqref{querrymapping}, we decompose the integration domain into a sufficiently large number of disjoint rectangles, analogous to Heinrich’s approach with congruent cubes in the Sobolev case \cite[Proof of Thrm. 1]{heinrich2003sobolev}. In the process we have to restrict the set of indices in the hyperbolic layer so that this selection can be assigned injectively to these rectangles.
	
	We give two different index sets depending on $p$ and $q$. We begin our investigation with the case $q > p$ and $d>1$: For $\ell \in \bbN$, 
	\begin{align}
		\ell\ge d\log(2d(d-1)),
		\label{l_sufficientlarge}
	\end{align}
	we set
	\begin{align}\label{defi:u}
		u\coloneq \mleft\lceil \frac{d-1}{d}\log \ell \mright\rceil.
	\end{align}    
	Then we divide the integration domain
	\begin{align*}
		[0,1)^{d} = \bigcup_{ \bh\in \bbB_u^d} D_{\bh}
	\end{align*}
	with $D_{\bh}\coloneq\prod_{i=1}^d[2^{-u}h_i,2^{-u}(h_i+1))$ into $2^{du}$ disjoint congruent cubes and set
	\begin{align*}
		\mcV_\ell 
		&\coloneq \mleft\{ \bj \in \bbN^d: \sum_{i=1}^d j_i= \ell \text{ and } j_i\ge u \text{ for } i =1,\dots, d \mright\}.
	\end{align*}
	With \eqref{l_sufficientlarge} we ensure $\ell\ge du$ such that $\mcV_\ell$ is not empty. Since
	\begin{align}
		| \mcV_\ell|
		&=\mleft| \mleft\{ \bj \in \bbN^d: \sum_{i=1}^d j_i= \ell-du+d \mright\}\mright|
		=\binom{\ell-du+d-1}{d-1}\notag \\
		&\le \ell^{d-1}
		= 2^{(d-1)\log \ell}\le 2^{du},
		\label{eq:anzV_l}
	\end{align}
	there is an injective mapping
	\begin{align*}
		h: \mcV_\ell \to \bbB_u^d, \bj \mapsto h(\bj).
	\end{align*}
	
	Let 
	\begin{align*}
		R_{\bj,\bk} \coloneq\prod_{i=1}^d \mleft[2^{-j_i}k_i,2^{-j_i}(k_i+1)\mright), \quad \bj \in \bbN_0^d, \bk \in \bbB_\bj
	\end{align*}
	be a half open dyadic rectangle. We restrict the set of admissible translation indices to
	\begin{align}
		\mcK_\bj \coloneq \{ \bk \in \bbB_\bj: R_{\bj,\bk} \subseteq D_{h(\bj)} \} \quad \text{and set} \quad \mcW_\ell \coloneq\{ (\bj,\bk): \bj \in \mcV_\ell, \bk \in \mcK_\bj \}.
		\label{eq:defK_jW_l}
	\end{align}
	An example is shown in Figure \ref{fig:rect}. For $(\bj,\bk),(\tilde \bj, \tilde \bk) \in \mcW_\ell, (\bj,\bk) \neq (\tilde \bj, \tilde \bk)$ we get
	\begin{align*}
		R_{\bj,\bk}\cap R_{\tilde \bj,\tilde \bk} = \emptyset,
	\end{align*} 
	and for all $\bj \in \mcV_\ell$ we have
	\begin{align*}
		D_{h(\bj)}
		=\bigcup_{\bk \in \mcK_\bj} R_{\bj,\bk}.
	\end{align*}

	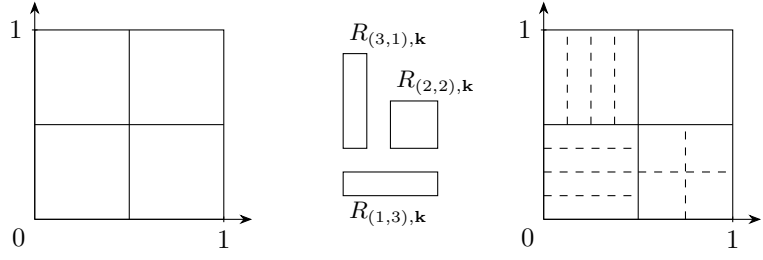
\begin{figure}
		\centering
		\begin{tikzpicture}[scale=2.5]
			\achsen
			\draw (0,0) rectangle (1,1);
			\draw (1/2,0) -- (1/2,1);
			\draw (0,1/2) -- (1,1/2);
		\end{tikzpicture}
		\begin{tikzpicture}[scale=2.5]
			\node[below left] at (0,0) {$\phantom{0}$};
			\node[below] at (1,-0.015) {$\phantom{1}$};
			\node[left] at (-0.015,1) {$\phantom{1}$};
			
			\draw (2/8,1/8) rectangle (6/8,2/8);
			\node[below right,font=\small,inner sep=2pt] at (2/8,1/8) {$R_{(1,3),\bk}$};
			
			\draw (4/8,3/8) rectangle (6/8,5/8);
			\node[above right,font=\small,inner sep=2pt] at (4/8,5/8) {$R_{(2,2),\bk}$};
			
			\draw (2/8,3/8) rectangle (3/8,7/8);
			\node[above right,font=\small,inner sep=2pt] at (2/8,7/8) {$R_{(3,1),\bk}$};
			
		\end{tikzpicture}
		\begin{tikzpicture}[scale=2.5]
			\achsen
			\draw (0,0) rectangle (1,1);
			\draw (1/2,0) -- (1/2,1);
			\draw (0,1/2) -- (1,1/2);
			
			\draw[dashed] (0,1/8) -- (1/2,1/8);
			\draw[dashed] (0,1/4) -- (1/2,1/4);
			\draw[dashed] (0,3/8) -- (1/2,3/8);
			
			\draw[dashed] (1/2,1/4) -- (1,0.25);
			\draw[dashed] (3/4,0) -- (3/4,0.5);
			
			\draw[dashed] (1/8,1/2) -- (1/8,1);
			\draw[dashed] (1/4,1/2) -- (1/4,1);
			\draw[dashed] (3/8,1/2) -- (3/8,1);
		\end{tikzpicture}
		\caption{Illustration for $q\ge p, d=2, \ell=4$: $u=1$, so we divide the domain into $2^{ud}=4$ squares. $\mcV_4=\{ (1,3),(2,2),(3,1) \}$, i.\,e. we have 3 types of rectangles. For each type, we choose exactly one of the sub-squares and consider all translations $\bk$ for which the corresponding rectangles are contained within the selected sub-square.}
		\label{fig:rect}
	\end{figure}
	
	The simple volume argument $\vol(D_{\bh(\bj)})=\sum_{\bk \in \mcK_\bj} \vol(R_{\bj,\bk})=|\mcK_\bj|2^{-\ell}$ reveals together with \eqref{defi:u}
	\begin{align}
		K_\ell 
		\coloneq |\mcK_\bj|=2^{\ell-du}
		\le 2^\ell \ell^{-(d-1)}.
		\label{eq:anzK_l}
	\end{align}
	The restricted index set $\mcK_\bj$ contains at most a fraction $\ell^{-(d-1)}$ of the indices in $\bbB_\bj$. To estimate the number of indices $\mcW_\ell$ we have to consider the set $\mcV_\ell$ again. With \eqref{defi:u} we get
	\begin{align*}
		du-d+1
		\le d\log \ell,
	\end{align*}	
	and thus an asymptotic lower bound for $\mcV_\ell$ by
	\begin{align*}
		\binom{\ell-du+d-1}{d-1}
		\gtrsim \ell^{d-1}.
	\end{align*}
	
	Together with \eqref{eq:anzV_l}, \eqref{eq:defK_jW_l} and \eqref{eq:anzK_l} there is a constant $c_1>0$
	\begin{align}\label{anzwl}
		c_1 2^\ell \le |\mcW_\ell|=\sum_{\bj \in \mcV_\ell} |\mathcal K_\bj|= |\mcV_\ell|K_\ell \le 2^\ell.
	\end{align}
	Now for $F \coloneq \{f \in L_\infty^{\mcW_\ell} : \|f\|_{L_\infty^{\mcW_\ell}}\le 1 \}$ we set
	\begin{align}\label{eq:gammaconst}
		\Gamma: F \to \BS, f\mapsto \sum_{{\bj,\bk}\in \mcW_\ell} \gamma(\beta(f_{\bj,\bk})) a_{\bj,\bk}
	\end{align}
	with $a_{\bj,\bk}$ defined in Proposition \ref{prop:normatomdecomb}, under the additional condition $\int \varphi(x) \diffd x =1$, and $\gamma$, $\beta$ defined in \eqref{eq:def_beta_gamma} for $m^*\in \bbN, m^*\ge2$.
	
	We proceed to the case $q\le p$ or $d=1$, where the construction is more straightforward. Let $\bj \in \bbN_0^d$ such that $|\bj|_1=\ell$. It suffices to adjust the index sets $\mcV_\ell=\{\bj\}$ and $\mcK_\bj=\bbB_{\bj}$, retaining all other notations. Observe that $|\mcW_\ell|=2^\ell$ in this instance.
	
	\subsection{Proof of Theorem \ref{thm:besovintlow}}
	\begin{proof}
		Let $\ell \in \bbN$. First we estimate $\| \Gamma f\|$ for $f \in F$ and $\Gamma$ defined in \eqref{eq:gammaconst}. By using  \eqref{distretizationerror} for $|z| \le 1, m^*\in \bbN, m^*\ge2$ we get
		\begin{align*}
			|\gamma(\beta(z))| \le |z|+|z-\gamma(\beta(z))|\le 1+2^{-m^*/2}\le 2,
		\end{align*}
		and with Proposition \ref{prop:normatomdecomb} for $f \in F$ it follows 
		\begin{align}
			\|\Gamma f\|_\BS 
			&\lesssim \mleft( \sum_{\bj \in \mcV_\ell }2^{|\bj|_1(r-1/p)q}\mleft( \sum_{\bk \in \mcK_{\bj}} |\gamma(\beta(f_{\bj,\bk}))|^p  \mright)^{q/p} \mright)^{1/q} \notag \\
			&\lesssim \mleft( \sum_{\bj \in \mcV_\ell } 2^{\ell(r-1/p)q}K_\ell^{q/p}   \mright)^{1/q}\notag\\
			&\lesssim 2^{\ell(r-1/p)}K_\ell^{1/p} |\mcV_\ell|^{1/q} \notag \\
			&\lesssim 2^{\ell r}\ell^{-(d-1)(1/p-1/q)_+}.
			\label{eq:estgamma}
		\end{align}
		So there is a constant $\tilde c>0$ such that $\Gamma (F) \subset \tilde c 2^{\ell r}\ell^{-(d-1)(1/p-1/q)_+}\sB(\BS)$. 
		We have 
		\begin{align*}
			I_d \Gamma f
			&= \sum_{(\bj,\bk) \in \mcW_\ell} \gamma ( \beta ( f_{\bj,\bk} )) \int_{\bbT^d} a_{\bj,\bk}(x)\diffd x 
			= \frac{1}{2^{\ell}} \sum_{(\bj,\bk) \in \mcW_\ell}  \gamma ( \beta ( f_{\bj,\bk})) 
			=\frac{|\mcW_\ell|}{2^\ell} S_{\mcW_\ell} (\gamma \circ \beta \circ f),
		\end{align*}
		where $|\mcW_\ell|/2^\ell\asymp 1$, see \eqref{anzwl}. The construction of $\Gamma$ allows us to define the required mappings $\eta$ and $\rho$ from Lemma \ref{lem:quantemb}. For some fixed $(\bj_0,\bk_0) \in \mcW_\ell$ let
		\begin{align*}
			\eta : \bbT^d \to \mcW_\ell, x \mapsto \begin{cases}
				(\bj_0,\bk_0) &\textrm{if } x \notin \bigcup_{(\bj,\bk) \in \mcW_\ell} R_{\bj,\bk},\\
				(\bj,\bk) &\textrm{if } x \in R_{\bj,\bk},
			\end{cases}
		\end{align*}
		and
		\begin{align*}
			\rho: \bbT^d \times \bbB_{m^*} \to \bbR, (x,z) \mapsto \gamma(z)a_{\eta(x)}(x).
		\end{align*}
		The supports of $a_{\bj,\bk}$ corresponding to different $(\bj,\bk)$ are disjoint. Therefore, for every $x$, at most one term in the following expression is nonzero and the sum simplifies to a single term
		\begin{align*}
			\Gamma(f)(x)
			&=\sum_{(\bj,\bk) \in \mcW_\ell} \gamma ( \beta ( f_{\bj,\bk})) a_{\bj,\bk}(x) 
			=  \gamma ( \beta ( f_{\eta(x)})) a_{\eta(x)}(x) 
			= \rho(x,\beta \circ f \circ \eta(x)).
		\end{align*}
		Consequently, $\Gamma$ is of the form \eqref{querrymapping}, allowing the application of Lemma \ref{lem:quantemb} with $\kappa = 1$ in the next step.
		
		For $n \in \bbN, n \ge 2$, define $\ell \coloneq \lceil \log(n/(cc_1))+1 \rceil$, where $c$ denotes the constant in Proposition \ref{prop:quantsum} and $c_1$ is the constant in \eqref{anzwl} if $q>p, d>1$, otherwise $c_1=1$. Without loss of generality, we choose $n$ sufficiently large to satisfy condition \eqref{l_sufficientlarge} for $\ell$ if $q>p, d>1$.
		
		We get $cc_1 2^{\ell-1} \le 2n \le c|\mcW_\ell|,$
		and with \eqref{eq:estgamma} and Proposition \ref{prop:quantsum} \ref{prop:quantsumiii} we have
		\begin{align*}
			\frac{1}{n} 
			&\lesssim \errquant_{2  n}(S_{\mcW_\ell}, F)
			\lesssim \errquant_{2  n}(S_{\mcW_\ell}\circ \gamma\circ\beta, F)+2^{-m^*/2}
			\lesssim \errquant_{2  n}(I_d\circ \Gamma, F)+2^{-m^*/2} \\
			&\lesssim 2^{\ell r}\ell^{-(d-1)(1/p-1/q)_+} \errquant_{n}(I_d, \sB(\BS))+2^{-m^*/2}\\
			&\lesssim n^r\log(n)^{-(d-1)(1/p-1/q)_+} \errquant_{n}(I_d, \sB(\BS))+2^{-m^*/2}.\\
		\end{align*}
		Since $m^*$ may be chosen arbitrarily large, we obtain the desired result.
	\end{proof}
	
	\section{Concluding remarks}
	Theorems \ref{thm:besovintup}, \ref{thm:besovintup_2}, and \ref{thm:besovintlow} show that for $1\le p,q \le \infty$, $\max\{1/p,2/q-1\}<r< 2$ and for $p=q=r=2$ the proposed quantum algorithm has the main decay rate $n^{-r-1}$ which is optimal up to logarithmic terms. Due to proof technique limitations, we only examined smoothness classes up to order 2 in our approach. The question remains how far the methods and results can be extended to higher degrees of smoothness. 
	
	Further, it remains open whether the logarithmic factors are intrinsic to the problem or an artifact of our analysis. At least the parameter $r$ in the $\log$-power for the upper bound looks somewhat artificial and can probably be dropped. We leave this for future contributions. It might require a different technique. This additional $\log$-term is one of the reasons why,  
	in contrast to isotropic Sobolev spaces with $p>2$ \cite{heinrich2003sobolev}, we were not able to observe the full $\mathcal{O}(1/n)$-gain between deterministic and quantum algorithms. The other reason is the appearance of the $\lambda_{p,q}(n)$-factor in case $\min\{p,q\}\le2$. Note that the deterministic rate of convergence is well-known
	\begin{align*}
		\errdet_n(I_d,\sB(\BS))\asymp n^{-r}(\log n)^{(d-1)(1-1/q)}\,,
	\end{align*}
	\cite[Thm. 8.5.1]{dung2018}. We believe that this full gain is present at least for $H^r_{\text{mix}}(\bbT^d) = S_{2,2}^rB(\bbT^d)$. It may be even possible to achieve a gain of more than $1/n$ compared to the deterministic case. More precisely, for $d>1$, does 
	\begin{align*}
		\errquant_n(I_d,\sB(H^r_{\text{mix}}(\bbT^d)))/\errdet_n(I_d,\sB(H^r_{\text{mix}}(\bbT^d))) = o(n^{-1})\
	\end{align*}
	hold true? We leave this as another open question. 
	
	A related effect has been observed by M. Ullrich \cite{ul17} for the randomized setting compared to the deterministic setting. In fact, in the randomized setting, sharp error estimates are known in the mixed Sobolev case $p=q=2$, namely
	\begin{align*}
		\errrand_n(I_d,\sB(H^r_{\text{mix}}(\bbT^d)))\asymp n^{-r-1/2}\,.
	\end{align*}
	(This result is also mentioned in the book by D\~ung, V. Temlyakov and T. Ullrich \cite[page 140]{dung2018}, however with a small typo in the right-hand side which occurs in case  $p>2$ for corresponding mixed Sobolev spaces.) Interestingly, there is no $\log$-term in this result and hence 
	\begin{align*}
		\errrand_n(I_d,\sB(H^r_{\text{mix}}(\bbT^d)))/\errdet_n(I_d,\sB(H^r_{\text{mix}}(\bbT^d))) \asymp n^{-1/2}(\log n)^{-(d-1)/2} = o(n^{-1/2})\,,
	\end{align*}
	for $d>1$. Due to the $\log$-term this is a purely multivariate effect (i.e., not present for $d=1$). 
	
	Furthermore, in contrast to randomized algorithms, the quantum error decay gain is roughly of order one for the entire range of the parameter $1\leq p\leq \infty$.   Randomized algorithms show the usual effect that the gain shrinks as $1-1/p$ as $p\to 1$. However, from the point of view of numerical integration the space $S^{2}_{1,\infty}B(\mathbb{T}^d)$ is particularly important since it is a smoothness space built upon $L_1$. This space contains tensor products of piecewise linear functions. Hence, in this regime we may gain {\em one} convergence order compared to the deterministic {\em and} randomized setting. In fact, we already observed for $p=1$ and $r<2$  a convergence order close to {\em three}. In any case, quantum might outperform random algorithms by more than order $1/2$, i.e., for $1<r<2$
	\begin{align*}
		\errquant_n(I_d,\sB(S_{1, \infty}^rB(\bbT^d)))/\errrand_n(I_d,\sB(S_{1,\infty}^rB(\bbT^d))) = o(n^{-1+\delta})
	\end{align*}
	for $\delta >0$. 
	Here we used that for $r>1$ 
	\begin{align*}
		\errrand_n(I_d,\sB(S_{1,\infty}^rB(\bbT^d))) 
		\asymp n^{-r}(\log n)^{(d-1)/2}\,,
	\end{align*}
	which will be proven in a forthcoming paper. A similar effect  has been already observed by Heinrich \cite[Page 12]{heinrich2001optimal}, \cite[Table 1]{heinrich2003sobolev} for isotropic Sobolev classes. Interestingly, our quantum lower bound shows that the full gain $1/n$ in this particular situation between quantum and random is not possible when $d>1$. In Heinrich's setting this might still be possible. However, the gain may be possible when comparing the quantum and deterministic setting.  
	
	The final point we want to make is that we only focused on query complexity in the present paper. The number of qubits, quantum gates, as well as the cost of classical post-processing is not included in our analysis. 
	
	\paragraph{Acknowledgment.} BK and TU acknowledge the support of the German Research Foundation (DFG 403/4-1). They also express their gratitude to Stefan Heinrich for extensive discussions, including those following his seminar presentation at TU Chemnitz in 2023 and those at MCQMC 2024 in Waterloo, Canada, and MCM 2025 in Chicago, USA, as well as for his insightful comments, observations, and supplementary references.
	{\small
		\printbibliography
	}
\end{document}